\documentclass[amssymb,12pt]{amsart}

\usepackage{latexsym}
\newtheorem{propo}{Proposition}[section]

\newtheorem{lemma}[propo]{Lemma}
\newtheorem{corol}[propo]{Corollary}

\newtheorem{theor}[propo]{Theorem}
\newtheorem{theorem}[propo]{Theorem}

\newcommand{\Irr}{{\mathrm {Irr}}}
\newcommand{\IBR}{{\mathrm {IBr}}}

\newcommand{\soc}{{\mathrm {soc}}}

\newcommand{\Char}{{\mathrm {char}}}

\newcommand{\ZZ}{{\mathbb Z}}
\newcommand{\NN}{{\mathbb N}}

\newcommand{\SSS}{{\sf S}}

\newcommand{\AAA}{{\sf A}}
\newcommand{\FF}{{\mathbb F}}

\newcommand{\GC}{\mathcal{G}}

\newcommand{\HC}{\mathcal{H}}
\newcommand{\LC}{\mathcal{L}}
\newcommand{\SC}{\mathcal{S}}
\newcommand{\TC}{\mathcal{T}}

\newcommand{\RC}{\mathcal{R}}
\newcommand{\PC}{\mathcal{P}}
\newcommand{\QC}{\mathcal{Q}}
\newcommand{\MB}{\mathbf{M}}

\newcommand{\lam}{\lambda}

\newcommand{\al}{\alpha}
\newcommand{\gam}{\gamma}
\newcommand{\om}{\varpi}

\newcommand{\LP}{{\Lambda^{+}}}

\newcommand\SL{\mathrm{SL}}
\newcommand{\Spin}{{\mathrm{Spin}}}

\begin{document}

\title[Representation Growth]
{Representation Growth in Positive Characteristic and 
Conjugacy Classes of Maximal Subgroups}

\author{Robert M. Guralnick}
\address{Department of Mathematics, University of Southern California,
Los Angeles, CA 90089-2532, USA}
\email{guralnic@math.usc.edu}

\author{Michael Larsen}
\address{Department of Mathematics\\
    Indiana University \\
    Bloomington, IN 47405\\
    U.S.A.}
\email{larsen@math.indiana.edu}

\author{Pham Huu Tiep}
\address{Department of Mathematics\\
    University of Arizona\\
    Tucson, AZ 85721\\
    U. S. A.} 
\email{tiep@math.arizona.edu}

\thanks{The authors gratefully acknowledge the support of the NSF 
(grants DMS-0653873, DMS-1001962, DMS-0800705, and DMS-0901241).}

\maketitle

\section{Introduction}

There have been a number of recent papers on the representation growth
of various families of groups.  In particular, see \cite{LL}, 
where very good estimates were given for the zeta function counting the 
complex irreducible representations of simple Lie groups. 
See also \cite{LS} for results for 
complex representations of finite simple groups. 
One gets similar bounds for cross characteristic representations
of Chevalley groups (using lower bounds for the minimal dimension of such
representations \cite{LaSe,SZ}).
Here we focus on projective modular representations of alternating groups
and Chevalley groups in the natural characteristic.   
In the latter case, we consider only restricted representations
(by Steinberg's tensor product theorem, this is the critical case).   

The first main result of the paper is the following:

\begin{theor}\label{main1}
{\sl Let $\FF$ be an algebraically closed field of characteristic $p$.

{\rm (i)} Let $\GC$ be a simple algebraic group over $\FF$ and let $R_{n}(\GC)$ be the 
number of restricted irreducible $\FF\GC$-representations of degree at most $n$. If
$p > 2$, then $R_{n}(\GC) \leq n^{3.8}$ when $\GC$ is of type $A_r$ and 
$R_{n}(\GC) \leq n^{2.5}$ otherwise. If $p = 2$, then $R_{n}(\GC) \leq n$.

{\rm (ii)} Let $G$ be a covering group of $\SSS_r$ or $\AAA_r$ with $r \geq 5$, and let 
$R_{n}(G)$ be the number of irreducible $\FF G$-representations of degree at most $n$. Then
$R_{n}(\GC) \leq n^{2.5}$.}
\end{theor}

The representation growth in the modular case is much more difficult to analyze than in 
the characteristic zero case.
For one thing, there is no known
dimension formula for modular irreducible representations  of $\SSS_r$.
For another, we do not know in general the dimensions of the
restricted irreducible  representations of
finite groups of Lie type 
(though there are a number of results in this direction). Hence, a key step in our
proof of Theorem \ref{main1} is to establish effective lower bounds on the dimension
of modular irreducible representations, see e.g. Theorem \ref{bound3} for the 
case of $\SSS_r$. 
  
While the bounds in Theorem \ref{main1}
are certainly not best possible, they are not so far off.  
For example, $\SL_2(p)$ has $d$ restricted representations of dimension up to $d$
(if $d \leq p$). For the complex simple Lie groups $\GC$, the linear bound
$R_n(\GC) \leq n$ for the number of complex irreducible representations 
has recently been proved in \cite{GLM}.  

We use Theorem \ref{main1} to obtain bounds for 
the number of conjugacy classes of maximal subgroups
for the almost simple groups with socle a finite classical group. 
If $G$ is a finite group, let $\mathcal{M}(G)$ denote the set of conjugacy classes 
of maximal subgroups of $G$, and let $m(G) =|\mathcal{M}(G)|$.   
We can then prove the following:

\begin{theor} \label{main3}  
{\sl Let $G$ be a finite almost simple group with socle a group 
of Lie type of rank $r$ defined over $\FF_q$.   
Then there are constants $a$ and $d$ such that
$m(G) < ar^6  + d r \log\log q$.}
\end{theor}

One should be able to replace $r^6$ by $r$.   The best previous results
were $O(r^r + \log\log q) $ (see the proof of \cite[Theorem 1.3]{LMS})
and $O(p^{r/3} + \log \log q)$ in \cite[\S7]{FG}, where $p$ is the unique
prime divisor of $q$.    Thus, Theorem \ref{main3}
is quite an improvement and much closer to the result about     conjugacy
classes of maximal (closed) subgroups of a semisimple Lie group 
(with the linear bound $O(r)$ in the rank $r$) recently proved in \cite{GLM}.  
Combining Theorem \ref{main3} with  \cite[Cor 5.3]{LMS}, we see that:

\begin{theor} \label{main4}  
{\sl If $G$ is a finite almost simple group, then $m(G) \le O( (\log |G|)^3)$.}
\end{theor}

Let $k(G)$ denote the number
of conjugacy classes of a finite group $G$.  This partially verifies 
(in a very strong way) a conjecture of Aschbacher and the first author \cite{AG}  
for almost simple groups:

\begin{corol} 
{\sl If $G$ is a finite almost simple group, then  $m(G)  \le O( (\log k(G))^6)$.}
\end{corol} 

The conjecture in \cite{AG} is that $m(G) < k(G)$ for all finite groups (the example of an elementary
abelian $2$-group shows this is best possible).  It was shown in \cite{AG} that the conjecture holds
for solvable groups.   This also implies various bounds on first cohomology groups.

\section{Type A}
Let $\FF$ be an algebraically closed field of characteristic $p > 0$ and 
let $\GC = \SL_{r+1}(\FF)$. Fixing a maximal torus in $\GC$, we consider 
the set of simple roots $\{ \al_1, \ldots, \al_r\}$ and the corresponding
set of fundamental weights $\{ \om_1, \ldots ,\om_r\}$ (in the usual
ordering). Then the set 
$$\LP = \left\{ \sum^{r}_{i=1}a_i \om_i \mid a_i \in \ZZ,a_i \geq 0
  \right\}$$
of dominant weights admits the 
partial ordering $\succ$ where $\lam \succ \mu$ precisely when 
$\lam - \mu = \sum^{r}_{i=1}k_i\al_i$ for some non-negative integers $k_i$.
Recall that the Cartan matrix expresses the simple roots in
terms of fundamental weights; in particular,
$$\al_1 = 2\om_1-\om_2,
  ~~\al_i = -\om_{i-1}+2\om_i-\om_{i+1} 
  \mbox{ for }2 \leq i \leq r-1,~~
  \al_r = -\om_{r-1}+2\om_r.$$
As usual, $W$ denotes the Weyl group, so $W \cong \SSS_{r+1}$. If 
$\lam \in \LP$, let $L(\lam)$ denote the irreducible $\FF\GC$-module
with highest weight $\lam$. 

We will rely on the following two results.

\begin{theor}\label{premet} {\rm \cite{Pr}}
{\sl Let $\GC$ be a simple, simply connected algebraic group in 
characteristic $p$. If the root system of $\GC$ has different root 
lengths, we assume that $p \neq 2$, and if $\GC$ is of type $G_2$ we 
also assume that $p \neq 3$. Let $\lam$ be a restricted dominant
weight. Then the set of weights of the irreducible $\GC$-module 
$L(\lam)$ is the union of the $W$-orbits of 
dominant weights $\mu$ with $\lam \succ \mu$.
\hfill $\Box$}
\end{theor}

\begin{lemma}\label{orb} {\rm \cite[Lemma 10.3B]{H}}
{\sl Let $\lam = \sum^r_{i=1}a_i\om_i$ be a dominant weight. Then the 
stabilizer of $\lam$ in the Weyl group is the Young subgroup 
generated by the reflections along the simple roots $\al_i$ for which 
$a_i = 0$.
\hfill $\Box$}
\end{lemma}

In this section we will produce a bound which is at most 
polynomial in $n$ for the total number $R_n = R_n(\GC)$ 
of restricted irreducible representations of $\GC$ of degree at most $n$.

\subsection{The large range: $n \geq (r+1)!$} 
First we derive a lower bound on $\dim L(\lam)$ for a {\it restricted} 
$\lam := \sum^r_{i=1}a_i\om_i \in \LP$. For any integers 
$x_i$ with $0 \leq x_i \leq \lfloor a_i/2 \rfloor$, $1 \leq i \leq r$, 
we can express $\gam := \sum^r_{i=1}x_i\al_i$ as
$\sum^{r}_{i=1}y_i\om_i$ with $y_i = -x_{i-1} + 2x_i -x_{i+1}$, where 
we define $x_{0} = x_{r+1} = 0$. In particular, $y_i \in \ZZ$ and 
$y_i \leq a_i$. It follows that   
$\mu := \lam - \gam$ is dominant and moreover
$\lam \succ \mu$. Hence, by Theorem \ref{premet} any such $\mu$
is a weight of $L(\lam)$. Assume in addition that $\mu \neq 0$. 
By Lemma \ref{orb}, the stabilizer of $\mu$ in $W = \SSS_{r+1}$ 
is a proper (Young) subgroup and so of index $\geq r+1$. Thus the 
$W$-orbit of $\mu$ has length $\geq (r+1)$. Since each $W$-orbit 
contains exactly one dominant weight, we have proved

\begin{lemma}\label{bound1}
{\sl Let $\lam := \sum^r_{i=1}a_i\om_i \in \LP$ be a restricted weight. 
Then
$$\dim L(\lam) \geq N(\lam) := 
  1 + (r+1)\left\{ \prod^{r}_{i=1}\left( 
  1 + \left\lfloor \frac{a_i}{2}\right\rfloor\right) -1 \right\}.$$ 
\hfill $\Box$}
\end{lemma}  
 
For any real $d \geq 1$, consider the function 
$$h(d) = 1 + \frac{1}{2} + \frac{1}{3} + \cdots + 
         \frac{1}{\lfloor d \rfloor}.$$ 
Since $1/k < \int^{k}_{k-1}dx/x$ for $k \geq 2$, we have 
\begin{equation}\label{for-h}
  h(d) < 1 + \int^{d}_{1}\frac{dx}{x} = 1 +\log d.
\end{equation}

Next we consider the function $g(r,d)$, which is the number of 
$r$-tuples $(c_1, \ldots ,c_r)$ of positive integers $x_i$ such
that $\prod^r_{i=1}c_i \leq d$; in particular, 
$g(1,d) = \lfloor d \rfloor$. We claim that 
\begin{equation}\label{for-g}
 g(r,d) \leq d\cdot h(d)^{r-1}.
\end{equation}
Indeed, we induct on $r$ and assume $r \geq 2$. Then 
$$g(r,d) = \sum^{\lfloor d \rfloor}_{j=1}g(r-1,d/j)
  \leq \sum^{\lfloor d \rfloor}_{j=1}\frac{d}{j} \cdot h\Bigl(\frac{d}{j}\Bigr)^{r-2} 
  \leq h(d)^{r-2} \cdot \sum^{\lfloor d \rfloor}_{j=1}\frac{d}{j} = 
  d \cdot h(d)^{r-1}.$$  

\begin{propo}\label{A-large}
{\sl Assume $n \geq (r+1)!$. Then for 
$\GC = \SL_{r+1}(\FF)$ we have 
$$R_{n}(\GC) \leq n^{3.4}/r^{3}.$$}
\end{propo}

\begin{proof}
1) By Lemma \ref{bound1}, $R_n$ is bounded above by the number of 
dominant weights $\lam = \sum^r_{i=1}a_i\om_i$ such that 
$N(\lam) \leq n$. Setting $c_i = 1 + \lfloor a_i/2 \rfloor$, 
we see that $1 \leq c_i \in \ZZ$ and 
$$\prod^r_{i=1}c_i \leq d := 1+(n-1)/(r+1).$$  
Also, $2c_i-2 \leq a_i \leq 2c_i-1$. Thus every $r$-tuple 
$(c_1, \ldots ,c_r)$ corresponds to at most $2^r$ distinct (restricted)
$r$-tuples $(a_1, \ldots ,a_r)$. Together with
(\ref{for-h}) and (\ref{for-g}), this implies that 
\begin{equation}\label{bound2}
  R_n \leq 2^r \cdot g(r,d) \leq 2^r \cdot d \cdot (1+ \log d)^{r-1}.
\end{equation}

Restricting $L(\lam)$ to a fundamental subgroup $\SL_2(\FF)$, 
we see that $a_i \leq \min(n-1,p)$; in particular we get the trivial
bound $R_n \leq \min(p^r,n^r)$. Since $n \geq (r+1)!$, the statement follows 
immediately when $r \leq 2$ or when $r = 3$ and $n > 3787$. 
On the other hand, if $r = 3$ and $24 \leq n \leq 500$, then 
$n^{3.4}/r^3 > 5000$ whereas $R_{500}(\GC) < 200$ by \cite{Lu}.
If $r = 3$ and $500 < n \leq 3787$, then (\ref{bound2}) implies
that $R_n \leq R_{3787} < 5 \cdot 10^5 < n^{3.4}/r^3$.  
Also when $r = 4$ and 
$120 \leq n < 720$ then $n^{3.4}/r^{3} > 10^5$ whereas 
$R_{719}(\GC) \leq 170$ by \cite{Lu}. So we will assume that $r \geq 4$
and $n \geq 720$.

2) By our assumption, $n \geq (r+1)! > ((r+1)/e)^{r+1}$. We claim that 
\begin{equation}\label{ratio}
  r+1 < (1.8) \cdot \frac{\log n}{\log \log n}
\end{equation} 
whenever $n \geq \max(6,(r+1)!)$ and $r \geq 1$. 
Indeed, the inequality is obvious if $r \leq 4$. We may now
assume that $r \geq 5$ and so $n \geq 720$. Observe that the function 
$f(x) = (\log x)/x$ is decreasing on $[e,\infty)$ and increasing on 
$(0,e]$, with maximum $1/e$ at $x = e$. 
Since $\log n \geq \log 720 > 6$, we see that  
$(\log n)/(\log \log n)$ is increasing as a function of $n \geq 720$. Hence it
suffices to prove (\ref{ratio}) for $n = (r+1)!$. Direct computation reveals
that (\ref{ratio}) holds for $n = (r+1)!$ and $5 \leq r \leq 69$
(in fact, the ratio between $11$ and $(\log 11!)/(\log \log 11!)$ is about 
$1.7989$). Now assume that (\ref{ratio}) does not hold for some 
$r \geq 70$. Writing $n = e^{e^x}$, we see that $x \geq \log \log 720 > 1$, and 
so $f(x) \leq 1/e$, i.e. $e (\log x) \leq x$. It follows that  
$\log n = e^x \geq x^e = (\log \log n)^e$, whence 
$$\log \log n > e \log \log \log n.$$
Writing 
$$r+1 = et \cdot \frac{\log n}{\log \log n},$$ 
we must have that $et \geq 1.8$ and 
$\log t > -0.4123$. But $n \geq 70! > 10^{100}$, hence 
$\log \log n > 5.4392$ and so 
$$(1.8)(1-\frac{1}{e} + \frac{\log t}{\log \log n}) > 1.$$   
Now 
$$\begin{array}{rl}(r+1) \cdot \log\dfrac{r+1}{e} & =   
  \dfrac{et \cdot \log n}{\log \log n} \cdot 
  (\log t + \log \log n - \log \log \log n)\\ 
  & = et \cdot \log n \cdot  
  \Bigl(1 - \dfrac{\log \log \log n}{\log \log n} + \dfrac{\log t}{\log \log n}\Bigr)\\ 
  \vspace{-2mm} & \vspace{-2mm} \\
  & > et \cdot \log n \cdot  
  \Bigl(1 - \dfrac{1}{e} + \dfrac{\log t}{\log \log n}\Bigr) \\
  & > \log n,
 \end{array}$$
and so $n < ((r+1)/e)^{r+1} < (r+1)!$, a contradiction. 

Recall that $r \geq 4$. Then $n \geq (r+1)! > e^r$, so 
$d = 1+(n-1)/(r+1) < n/r < n/e$, whence $1 + \log d < \log n$
and also $r < \log n$. It then follows by (\ref{ratio}) that
$$(1+ \log d)^{r-1} < (\log n)^{r-1} <  
  (\log n)^{\frac{(1.8)\log n}{\log \log n} -2} < n^{1.8}/r^2.$$
Also, since $r \geq 4$ we have $2^r < ((r+1)!)^{0.6} \leq n^{0.6}$.
Consequently,
$$R_n  \leq 2^r \cdot d \cdot (1+ \log d)^{r-1}
  < n^{0.6} \cdot \frac{n}{r} \cdot \frac{n^{1.8}}{r^{2}}  
  = \frac{n^{3.4}}{r^{3}}.$$
\end{proof}

Of course when $n \geq (r+1)!$, the term $r^3$ in the bound $n^{3.4}/r^3$ is
negligible. Our next goal is to prove the upper bound $n^{3.8}$ for all $n$.
  
\subsection{A generating function estimate}
Let $\NN$ denote the set of non-negative integers and
$$k(r,s) := |\{(x_1,\ldots,x_r)\in \NN^r \Bigm| 
\sum_{i=1}^r\min(i,r+1-i) \cdot x_i = s\}|.$$
We can regard $k(r,s)$ as the $x^s$-coefficient in the power series
$$(1-x)^{-1}(1-x^2)^{-1}(1-x^3)^{-1}\cdots(1-x^2)^{-1}(1-x)^{-1}.$$
As $(1-x^k)^{-1}-1$ has non-negative coefficients, $k(r,s)$ is bounded 
above by the
$x^s$ coefficient in
$$\Bigl(\prod_{i=1}^\infty (1-x^i)^{-1}\Bigr)^2 = 
\Bigl(\sum_{n=0}^\infty p(n) x^n\Bigr)^2,$$
where $p(n)$ denotes the partition function.
Thus
$$k(r,s) \le \sum_{i=0}^s p(i)p(s-i).$$
From the classical theory of the partition function \cite[Th.~14.5]{Ap}, 
it is well known that
\begin{equation}\label{pi}
  p(n) < e^{\pi\sqrt{2n/3}}.
\end{equation}
By the concavity of $\sqrt x$,
$$k(r,s) < (s+1)e^{2\pi\sqrt{s/3}}.$$
It follows that

\begin{lemma}\label{part}
{\sl In the above notation,
$$\sum_{s=0}^{N}k(r,s) < 
  \frac{(N+1)(N+2)}{2}e^{2\pi\sqrt{N/3}}.$$
\hfill $\Box$}
\end{lemma}

\subsection{Weight distributions in simple modules}
Throughout this subsection we set $k := \lfloor (r-1)/2 \rfloor$ 
so that $r \in \{2k+1,2k+2\}$.
For any $\lam = \sum^{r}_{i=1}a_i\om_i \in \LP$ we define 
$$[\lam] := \sum^{r}_{i=1}\min(i,r+1-i) \cdot a_i.$$
Note that $[\al_i] = 0$ unless 
$i \in \{ k+1,r-k \}$. 

\begin{lemma}\label{incr}
{\sl Let $\lam = \sum^{r}_{i=1}a_i\om_i \in \LP$ and $1 \leq m \leq k$. 
Assume that $\sum^{m}_{i=1}ia_i > m$. Then there exists 
$\mu  = \sum^{r}_{i=1}b_i\om_i \in \LP$ with $[\lam] = [\mu]$, 
$\lam \succ \mu$, $b_i = a_i$ for all $i \geq m+2$, and 
$b_{m+1} > a_{m+1}$.}
\end{lemma}

\begin{proof}
Let $j$ be the largest index such that $1 \leq j \leq m$ and $a_j > 0$.
Such $j$ exists since $\sum^{m}_{i=1}ia_i > m > 0$. We prove the lemma
by induction on $m-j$. Note that if 
$\mu = \lam - \sum^{m}_{i=1}k_i\al_i = \sum^r_{i=1}b_i\om_i$ with
$k_i \in \NN$, then $[\lam] = [\mu]$, $\lam \succ \mu$, and 
$b_i = a_i$ for $i \geq m+2$. 

1) For the induction base, consider the case $j = m$, i.e. $a_m > 0$. 
If $a_m \geq 2$, then $\mu := \lam-\al_m \in \LP$ has the desired 
properties, as $\al_m = -\om_{m-1}+2\om_m-\om_{m+1}$. Assume $a_{m} = 1$.
Since $\sum^{m}_{i=1}ia_i > m$, there must be some $i$ with $a_i \geq 1$ 
and $1 \leq i \leq m-1$. Then 
$\beta = \al_i + \al_{i+1} + \cdots + \al_{m}$ equals 
$-\om_{i-1}+\om_i+\om_m-\om_{m+1}$ (where we define $\om_0 := 0$). It 
follows that $\mu := \lam - \beta \in \LP$ has the desired properties.

2) For the induction step, assume that $a_j > 0$, but 
$a_{j+1} = a_{j+2} = \cdots = a_m = 0$ for some $1 \leq j \leq m-1$. 
Assume $a_j \geq 2$. Setting 
$\nu := \lam - \al_j = \sum^r_{i=1}a'_i\om_i$, we see that
$a'_i$ equals $a_i+1$ if $i = j-1 > 0$ or $i = j+1$, $a_i -2$ if $i = j$,
and $a_i$ otherwise. Thus $\nu \in \LP$, $[\lam] = [\nu]$, 
$\lam \succ \nu$, $\sum^m_{i=1}ia'_i = \sum^m_{i=1}ia_i > m$, and 
$a'_{j+1} = 1$. Applying the induction hypothesis to $\nu$ instead of 
$\lam$, we see that the desired $\mu$ exists. 

Now we may assume that $a_j = 1$. Since 
$\sum^{m}_{i=1}ia_i > m$, there must be some $i$ with $a_i \geq 1$ 
and $1 \leq i \leq j-1$. Then 
$\gam = \al_i + \al_{i+1} + \cdots + \al_j$ equals 
$-\om_{i-1}+\om_i+\om_j-\om_{j+1}$. As above,
we can check that $\nu \in \LP$, $[\lam] = [\nu]$, 
$\lam \succ \nu$, $\sum^m_{i=1}ia'_i = \sum^m_{i=1}ia_i > m$, and 
$a'_{j+1} = 1$, $a'_{i} = a_i$ for $i \geq m+2$, for 
$\nu := \lam -\gam = \sum^r_{i=1}a'_i\om_i$. Hence we can apply the
induction hypothesis to $\nu$ to obtain the desired $\mu$.
\end{proof}

\begin{lemma}\label{middle}
{\sl Let $\lam = \sum^{r}_{i=1}a_i\om_i \in \LP$ and $1 \leq m \leq k$. 
Assume that $a_{k+1} \geq 2m+1$ if $r = 2k+1$ and 
$a_{k+1}+a_{k+2} \geq 2m+3$ if $r = 2k+2$. Then there exists a dominant 
weight $\mu = \sum^{r}_{i=1}b_i\om_i$ with $\lam \succ \mu$ and 
$b_i > 0$ for $k-m+1 \leq i \leq r-k+m$.}
\end{lemma}

\begin{proof}
1) First we consider the case $r = 2k+1$. 
For any $j$ between $2$ and $k$, notice that 
$$\al_{j} + \al_{j+1} + \cdots + \al_{2k+2-j}
 = -\om_{j-1} + \om_j + \om_{2k+2-j} - \om_{2k+3-j}.$$
Hence, choosing 
$$\begin{array}{rl}
  \beta := &   
  m\al_{k+1} + (m-1)(\al_{k} + \al_{k+1} + \al_{k+2})\\ 
  & + (m-2)(\al_{k-1} + \al_{k} + \al_{k+1} + \al_{k+2} + \al_{k+3})\\
 & + \cdots + 
   1 \cdot (\al_{k-m+2} + \al_{k-m+3} + \cdots + \al_{k+m-1} + \al_{k+m}),
 \end{array}$$
we have $\beta = (2m+1)\om_{k+1}-\sum^{k+m+1}_{i=k-m+1}\om_i$, and so 
$\mu := \lam-\beta$ has the desired properties (as $k+m+1 = r-k+m$).  

2) Now we consider the case $r = 2k+2$. First suppose that 
$a_{k+1},a_{k+2} \geq m+1$. For any $j$ between $2$ and $k+1$, notice that 
$$\al_{j} + \al_{j+1} + \cdots + \al_{2k+3-j}
 = -\om_{j-1} + \om_j + \om_{2k+3-j} - \om_{2k+4-j}.$$
Hence, choosing 
$$\begin{array}{rl}
  \beta := &   
  m(\al_{k+1}+\al_{k+2}) + (m-1)(\al_{k} + \al_{k+1} + \al_{k+2} + \al_{k+3})\\ 
  & + (m-2)(\al_{k-1} + \al_{k} + \al_{k+1} + \al_{k+2} + \al_{k+3} + \al_{k+4})\\
 & + \cdots + 
   1 \cdot (\al_{k-m+2} + \al_{k-m+3} + \cdots + \al_{k+m} + \al_{k+m+1}),
 \end{array}$$
we have $\beta = (m+1)(\om_{k+1}+\om_{k+2})-\sum^{k+m+2}_{i=k-m+1}\om_i$, and so 
$\mu := \lam-\beta$ has the desired properties (as $k+m+2 = r-k+m$).  

Finally, assume that $a_{k+2} \leq m$ for instance. Setting 
$s := m+1-a_{k+2} \geq 1$, for  
$$\nu := 
  \lam - s\al_{k+1}-(s-1)\al_{k} - \cdots - 1 \cdot \al_{k-s+2} = 
  \sum^r_{i=1}c_i\om_i$$ 
we have $\nu \in \LP$, $\lam \succ \nu$, $c_{k+2} = m+1$, and 
$c_{k+1} = (a_{k+1}+a_{k+2}) -(m+2) \geq m+1$.
Now we can apply the previous case to $\nu$. 
\end{proof}

\begin{propo}\label{m-good}
{\sl Let $\lam = \sum^{r}_{i=1}a_i\om_i \in \LP$ and $1 \leq m \leq k$. 
Assume that 
$$[\lam] \geq \left\{ \begin{array}{rl}
  2m(k+1)+2k+1, & \mbox{if }r = 2k+1,\\
  (2m+2)(k+1)+2k+1, & \mbox{if }r = 2k+2. \end{array} \right.$$ 
Then there exists a dominant weight 
$\mu = \sum^{r}_{i=1}b_i\om_i$ with $\lam \succ \mu$ and 
$b_i > 0$ for $k-m+1 \leq i \leq r-k+m$.}
\end{propo}

\begin{proof}
Assume the contrary. Among all counterexamples 
$\lam = \sum^r_{i=1}a_i\om_i$ with given 
$N = [\lam]$ (there are only finitely many dominant weights 
$\delta$ with $[\delta] = N$), choose one with {\it largest possible} 
$t(\lam)$, where $t(\lam) := a_{k+1}$ if $r = 2k+1$ and
$t(\lam) := a_{k+1}+a_{k+2}$ if $r = 2k+2$. Since $\lam$ is a 
counterexample, there is no dominant weight $\mu \prec \lam$ with the 
prescribed properties. Hence by Lemma \ref{middle}, 
$t(\lam) \leq 2m$ if $r = 2k+1$, and $t(\lam) \leq 2m+2$ if $r = 2k+2$.
Observe that 
$$[\lam] = \sum^{k}_{i=1}ia_i + \sum^{r}_{i=r-k+1}(r+1-i)a_i + (k+1)t(\lam).$$
It follows that $\sum^{k}_{i=1}ia_i + \sum^{r}_{i=r-k+1}(r+1-i)a_i \geq 2k+1$. 
By symmetry we may assume that 
$\sum^k_{i=1}ia_i > k$. By Lemma \ref{incr} (in which we now set $m := k$), 
there is some $\nu = \sum^r_{i=1}c_i\om_i \in \LP$ with $[\nu] = [\lam] = N$,
$\lam \succ \nu$, $c_{k+1} > a_{k+1}$, and $c_i = a_i$ for $i \geq k+2$;
in particular, $t(\nu) > t(\lam)$. By the choice of $\lam$, 
$\nu$ cannot be a counterexample, which means that there exists 
a dominant weight $\mu = \sum^{r}_{i=1}b_i\om_i$ with $\nu \succ \mu$ and 
$b_i > 0$ for $k-m+1 \leq i \leq r-k+m$. But in this case
$\lam \succ \mu$, a contradiction.   
\end{proof}

We will need a variant of Proposition \ref{m-good} for $m = 0$;

\begin{lemma}\label{middle2}
{\sl Let $\lam = \sum^{r}_{i=1}a_i\om_i \in \LP$ with 
$[\lam] \geq 2k+1$. Then there exists a dominant weight 
$\mu = \sum^{r}_{i=1}b_i\om_i$ with $[\lam] = [\mu]$, $\lam \succ \mu$, and 
$b_i > 0$ for some 
$i \in \{ k+1,r-k\}$.}
\end{lemma}

\begin{proof}
Consider the case $r = 2k+2$ for instance and assume that 
$a_{k+1} = a_{k+2} = 0$. Since $[\lam] \geq 2k+1$, by symmetry we may assume
that $\sum^{k}_{i=1}ia_i \geq k+1$. Now we can apply Lemma \ref{incr} with
$m := k$. 
\end{proof}
 
By Lemma \ref{orb}, if all the cofficients $b_i$ of the 
weight $\mu = \sum^r_{i=1}b_i\om_i \in \LP$ are {\it positive}, then the
$W$-orbit of $\mu$ has length $(r+1)!$. We will call any such dominant
weight {\it good}. 

\begin{corol}\label{good}
{\sl Let $\lam = \sum^{r}_{i=1}a_i\om_i \in \LP$. Assume that 
$[\lam] \geq (r^2+2r-2)/2$. Then there is a good weight 
$\mu \in \LP$ with $\lam \succ \mu$.}
\end{corol}

\begin{proof}
Apply Proposition \ref{m-good} with $m = k$.
\end{proof}

Corollary \ref{good} yields the following exponential (in $r$) upper bound
for $R_n$:
 
\begin{corol}\label{A-exp}
{\sl Assume $n < (r+1)!$. Then for 
$\GC = \SL_{r+1}(\FF)$ we have 
$$R_{n}(\GC) < f_1(r) := \frac{(r+1)^{4}}{8} 
  e^{2\pi \sqrt{\frac{r^{2}}{6} + \frac{r}{3} - \frac{1}{2}}}.$$
Moreover, $R_n(\GC) \leq n^{3.8}$ if $1 \leq r \leq 10$ in addition.}
\end{corol}

\begin{proof}
The statements are obvious for $r = 1$ or $n \leq r$
(since $\SL_{r+1}(\FF)$ has no irreducible representations of degree between
$2$ and $r$, see e.g. \cite{Lu}), 
so we may assume that $r \geq 2$ and $n \geq r+1$. 
For the first statement, 
we count the total number of restricted dominant weights
$\lam = \sum^r_{i=1}a_i\om_i$ with $\dim L(\lam) < (r+1)!$. Since 
the $W$-orbit of any good weight has length $(r+1)!$, this bound implies 
that $L(\lam)$ does not afford any good weight. Applying 
Premet's Theorem \ref{premet} and Corollary \ref{good}, we see
that $[\lam] \leq N := (r^2+2r-3)/2$. Clearly, the number of $\lam$ 
satisfying this condition is at most $\sum^{N}_{s=0}k(r,s)$, so
the first statement follows from Lemma \ref{part}. 

Now assume that $2 \leq r \leq 10$. Then the restricted irreducible 
representations of $\GC$ up to a certain degree $B(r) \geq (r+1)^4$ are listed in 
\cite{Lu} (as well as online on Frank L\"ubeck's website), 
and the total number of these representations is 
less than $(r+1)^{2.7}$. In particular, we are done if 
$n \leq B(r)$ (since $n \geq r+1$), or if $2 \leq r \leq 5$ 
(since $B(r) > (r+1)!$ for these $r$). Suppose that $6 \leq r \leq 10$ and
$B(r) < n < (r+1)!$. In this case, one can check that
$f_1(r) \leq B(r)^{3.8}$, whence $R_n \leq n^{3.8}$.         
\end{proof}

\subsection{The mid-range: $ \binom{r+1}{k+1} \leq n < (r+1)!$} 
Here we aim to prove the bound $R_n < n^{3.8}$ in the mid-range, that is, for
$$(r+1)! > n \geq d_1 := \binom{r+1}{k+1},$$
where $k = \lfloor (r-1)/2 \rfloor$ as before.
Notice that $d_1 \geq 2^{r+1}/(r+1)$. By Corollary \ref{A-exp},
$R_n \leq f_{1}(r)$, and direct computation shows that
$f_1(r) < (d_1)^{3.8}$ when $r \geq 730$. In particular, 
$R_n < n^{3.8}$ in the mid-range for $r \geq 730$; also we are done 
by Corollary \ref{A-exp} if $r \leq 10$. So we will now assume that 
$11 \leq r < 730$.

\bigskip
{\bf Case 1: $19 \leq r < 730$.} We set $m := \lfloor r/6 \rfloor$ 
(in particular $m \geq 3$), and distinguish
two subcases: {\bf (1a)}, where for some $\lam$ with 
$\dim L(\lam) \leq n$, there is a dominant weight 
$\mu = \sum^{r}_{i=1}b_i\om_i$ such that $\lam \succ \mu$ and $b_i > 0$ whenever
$k-m+1 \leq i \leq r-k+m$, and {\bf (1b)} otherwise.

\smallskip
Suppose we are in the subcase {\bf (1a)}. By Theorem \ref{premet}, $\mu$ is a 
weight in $L(\lam)$. By Lemma \ref{orb}, the stabilizer of
$\mu$ in $W$ is contained in the Young subgroup generated by the reflections
along the simple roots $\al_i$ with $1 \leq i \leq k-m$ and 
$\al_j$ with $r-k+m+1 \leq j \leq r$, which is isomorphic to
$\SSS_{k-m+1} \times \SSS_{k-m+1}$. It follows that
$$\dim L(\lam) \geq \frac{(r+1)!}{((k-m+1)!)^2}
    = \left( \prod^{k-m+1}_{i=1} \frac{k-m+1+i}{i}\right)
      \cdot \prod^{r+1}_{j=2k-2m+3}j.$$      
Since $k-m+1 \geq 5$, it is easy to see that 
$$\prod^{k-m+1}_{i=1} \dfrac{k-m+1+i}{i} > 2^{k-m+2} \geq 2^{r-(2k+1)+2m} 
  \geq 2^{2m},$$ 
and that $2j \geq r+2$ for $2k-2m+3 \leq j \leq r+1$. Hence 
$$n \geq \dim L(\lam) > (r+2)^{2m} = (r+2)^{2\lfloor r/6 \rfloor} =: d_2.$$ 
Since $r \geq 19$, one can check that $f_1(r) < (d_2)^{3.76}$ and so
$R_n < n^{3.76}$.

\medskip
Next suppose that we are in the subcase {\bf (1b)}. It then follows by
Proposition \ref{m-good} (for $m = \lfloor r/6 \rfloor$) that 
$$[\lam] \leq \left\{ \begin{array}{ll}
  2m(k+1)+2k \leq N := (r^{2}-1)/6+r, & \mbox{if }r = 2k+1,\\
  (2m+2)(k+1)+2k \leq N:= r^{2}/6+2r-2 , & \mbox{if }r = 2k+2. 
  \end{array} \right.$$ 
whenever $\dim L(\lam) \leq n$. Applying Lemma \ref{part} we obtain 
$$R_{n} \leq \sum^{N}_{s=0}k(r,s) < f_2(r) := \left\{ \begin{array}{ll}
  \dfrac{1}{2} \cdot 
  \left(\dfrac{r^2+11}{6}+r\right)^{2} \cdot 
  e^{2\pi \sqrt{\frac{r^{2}-1}{18} + \frac{r}{3}}}, & \mbox{if }r = 2k+1,\\
  \dfrac{1}{2} \cdot 
  \left(\dfrac{r^2}{6}+2r\right)^{2} \cdot 
  e^{2\pi \sqrt{\frac{r^{2}}{18} + \frac{2r-2}{3}}}, & \mbox{if }r = 2k+2.
  \end{array} \right.$$
For $r \geq 19$, one can check that 
$f_2(r) < (d_1)^{3.76}$ and so $R_n < n^{3.76}$.

\bigskip
{\bf Case 2: $11 \leq r < 19$.} Now we set $m := 2$ and 
consider two subcases: {\bf (2a)}, where for some $\lam$ with 
$\dim L(\lam) \leq n$, there is a dominant weight 
$\mu = \sum^{r}_{i=1}b_i\om_i$ such that $\lam \succ \mu$ and $b_i > 0$ whenever
$k-m+1 \leq i \leq r-k+m$, and {\bf (2b)} otherwise.

\smallskip
Suppose we are in the subcase {\bf (2a)}. By Theorem \ref{premet}, $\mu$ is a 
weight in $L(\lam)$. By Lemma \ref{orb}, the stabilizer of
$\mu$ in $W$ is contained in the Young subgroup generated by the reflections
along the simple roots $\al_i$ with $1 \leq i \leq k-2$ and 
$\al_j$ with $r-k+3 \leq j \leq r$, which is isomorphic to
$\SSS_{k-1} \times \SSS_{k-1}$. Since $k \geq 5$, it follows that
$$n \geq \dim L(\lam) \geq \frac{(r+1)!}{((k-1)!)^2} =: d_3.$$
As $11 \leq r \leq 19$, one readily checks that $f_1(r) < (d_3)^{2.9}$ and so
$R_n < n^{2.9}$.

\medskip
Next suppose that we are in the subcase {\bf (2b)}. It then follows by
Proposition \ref{m-good} (for $m = 2$) that 
$$[\lam] \leq \left\{ \begin{array}{ll}
  6k+4 = 3r+1, & \mbox{if }r = 2k+1,\\
  8k+6 = 4r-2, & \mbox{if }r = 2k+2, 
  \end{array} \right.$$ 
whenever $\dim L(\lam) \leq n$. Applying Lemma \ref{part} we obtain 
$$R_{n} \leq \sum^{4r-2}_{s=0}k(r,s) < f_3(r) 
  := 8r^2 \cdot 
  e^{2\pi \sqrt{(4r-2)/3}}.$$
For $11 \leq r \leq 18$ and $n \geq (r+1)^4$, one can check that 
$R_n \leq f_3(r) < n^{3.29}$. On the other hand, if $n < (r+1)^4$ (and 
$r \leq 20$), all the restricted irreducible representations of $\GC$ are 
listed on L\"ubeck's website, from which one can also see that 
$R_n < (d_1)^{3.29} \leq n^{3.29}$.

\medskip
Thus we have proved

\begin{propo}\label{A-middle}
{\sl Assume 
$d_1 := \binom{r+1}{k+1}\leq n < (r+1)!$. 
Then for $\GC = \SL_{r+1}(\FF)$ we have $R_n(\GC) \leq n^{3.8}$.
\hfill $\Box$}
\end{propo}

\subsection{The small range: $n < \binom{r+1}{k+1}$}
For $n < d_1 = \binom{r+1}{k+1}$ 
and among all restricted dominant weights $\lam$ with
$\dim L(\lam) \leq n$, we choose $\lam = \lam_0 = \sum^{r}_{i=1}a_i\om_i$ 
with largest possible $[\lam_0]$. Observe that 
$$[\lam_0] \leq 2k.$$
For, assume the contrary. Then by Lemma \ref{middle2} and Theorem 
\ref{premet}, the module $L(\lam_0)$ contains a dominant weight 
$\mu = \sum^{r}_{i=1}b_i\om_i$ with $b_i > 0$ for some 
$i \in \{ k+1,r-k \}$. By Lemma
\ref{orb}, the stabilizer of $\mu$ in $W$ is contained in the 
Young subgroup $\SSS_{i} \times \SSS_{r+1-i}$ generated by the reflections 
along the simple roots $\al_j$ with $j \neq i$ which has index $d_1$ 
in $W = \SSS_{r+1}$. Thus 
$n \geq \dim L(\lam_0) \geq d_1$, contrary to our assumptions. 

Now we set $m := \lfloor ([\lam_0]+1)/2 \rfloor$; in particular,
$0 \leq m \leq k$ and $2m \geq [\lam_0] \geq 2m-1$. By the choice of 
$\lam_0$, we have 
\begin{equation}\label{small1}
  [\lam] \leq 2m
\end{equation}
for all restricted $\lam$ with $\dim L(\lam) \leq n$. If $m = 0$,
then $n=1$ and $R_n = n$. We will therefore assume that $m \geq 1$. 
We claim
that the module $L(\lam_0)$ admits a dominant weight 
$\mu = \sum^{r}_{i=1}b_i\om_i$ with $b_j > 0$ for some $j$ in the range 
$m \leq j \leq r+1-m$. (Indeed, we are done if $a_j > 0$ for some such 
index $j$. Otherwise we must have $m \geq 2$ and 
$a_{m} = a_{m+1} = \cdots = a_{r+1-m} = 0$. Since 
$[\lam_0] \geq 2m-1$, by symmetry we may assume that 
$\sum^{m-1}_{i=1}ia_i \geq m$. Applying Lemma \ref{incr} to $\lam_0$, we obtain
a dominant weight $\mu = \sum^{r}_{i=1}b_i\om_i \prec \lam_0$ with 
$b_{m} > a_{m} = 0$, as desired.) Now by Lemma \ref{orb} the stabilizer of
$\mu$ in $W$ is contained in the Young subgroup $\SSS_{j} \times \SSS_{r+1-j}$ 
generated by the reflections along the simple roots $\al_i$ with $i \neq j$,
which has index $\binom{r+1}{j}$.
Recall that $m \leq j \leq r+1-m$ and $1 \leq m \leq k \leq (r-1)/2$. 
It follows that
$$\binom{r+1}j\geq 
  \binom{r+1}m\geq 2^{m+1}.$$
Thus we have shown that    
\begin{equation}\label{small2}
  n \geq \dim L(\lam_0) \geq  
  \binom{r+1}{m}\geq 2^{m+1} =: d_4.
\end{equation}

Now (\ref{small1}) and Lemma \ref{part} imply that
$$R_n \leq \sum^{2m}_{s=0}k(r,s) < f_{4}(m) := 
  2(m+1)^{2} \cdot e^{2\pi\sqrt{2m/3}}.$$
Since $m \leq \log_{2}n-1$ because of (\ref{small2}), we obtain the following
sublinear bound for $R_n$ in the small range $n < d_1$:
$$R_n < 2(\log_{2}n)^{2} \cdot 
  e^{2\pi\sqrt{2(\log_{2}n-1)/3}}.$$
If $m$ is large enough, say $m \geq 80$, then one can check
that $f_4(m) < d_4$, whence $R_n < n$. Direct computation also reveals that
when $m \geq 6$ we have $f_4(m) < (d_4)^{3.8}$ and so $R_n < n^{3.8}$. By 
Corollary \ref{A-exp}, the same is true if $r \leq 10$. 

Assume $r \geq 11$. If $m = 5$, then (\ref{small2}) also implies 
$n \geq  \binom{12}5 = 792$,
and so $R_n \leq f_4(5) < n^{2.44}$. Similarly, for $1 \leq m \leq 4$ 
by (\ref{small2}) we have 
$n \geq  \binom{12}m$,
and so $R_n \leq f_4(m) < n^{3.11}$. Thus we have proved

\begin{propo}\label{A-small}
{\sl Assume 
$1 \leq n < d_1 = \binom{r+1}{k+1}$. 
Then for $\GC = \SL_{r+1}(\FF)$ we have $R_n(\GC) \leq n^{3.8}$. 
In fact we also have
$$R_n < 2(\log_{2}n)^{2} \cdot 
  e^{2\pi\sqrt{2(\log_{2}n-1)/3}}.$$
\hfill $\Box$}
\end{propo}   
 
\subsection{The general case}
Combining Propositions \ref{A-large}, \ref{A-middle}, and \ref{A-small}, we 
get:

\begin{theor}\label{mainA}
{\sl The number $R_{n}(\GC)$ of restricted irreducible 
$\FF\GC$-representations of dimension $\leq n$ of $\GC = \SL_{r+1}(\FF)$ is at most
$n^{3.8}$.
\hfill $\Box$}
\end{theor}

For future reference, we prove the following strengthening of
Theorem \ref{mainA} in the case $r = 5$. 

\begin{propo}\label{A5}
{\sl For $\GC = \SL_{6}(\FF)$ we have $R_n(\GC) \leq n^{2.5}$.}
\end{propo}

\begin{proof}
1) Using \cite{Lu}, one readily checks that $R_n \leq n$ for $n \leq 2500$.
Also, setting $d := 1 +(n-1)/6$, by (\ref{bound2}) we have 
$$R_n \leq f(d) := 2^5 \cdot d \cdot (1+\log d)^{4}.$$
Since $f(d) < n^{2.5}$ for $n \geq 57,750$, we may now assume that 
$2500 < n < 57,750$. 
Among all restricted $\lam$ with $\dim L(\lam) \leq n$,
choose $\lam_0$ with largest possible $[\lam_0] =:N$. 

Assume that $N \leq 76$. In this case, $R_n$ is at most the total number
of $5$-tuples $(x_1,x_2,x_3,x_4,x_5) \in \NN^5$ with 
$x_1+2x_2+3x_3+2x_4+x_5 \leq 76$, which is $2,415,231$, stricly less 
than $2500^{2.5} < n^{2.5}$. 

2) We may now assume that $N \geq 77$. First we claim that, for any dominant 
weight $\lam = \sum^{5}_{i=1}a_i\om_i$ with $[\lam] \geq 77$, there exists a 
dominant weight $\mu = \sum^{5}_{i=1}b_i\om_i$ with $\lam \succ \mu$ and 
$b_3 \geq 25$. Assuming the contrary, among all the counterexamples
$\lam$ with given $L := [\lam]$, fix one with largest possible $a_3$. 
Then $a_3 \leq 24$. Since 
$$[\lam] = (a_1+2a_2) + 3a_3 +(2a_4+a_5) \geq 77,$$
by symmetry we may assume that $a_1+2a_2 \geq 3$. Applying Lemma \ref{incr}
with $m = k = 2$ to $\lam$, we obtain a dominant weight 
$\nu = \sum^{5}_{i=1}c_i\om_i$ with $[\lam] = [\nu]$ and $c_3 > a_3$. By
the choice of $\lam$, $\nu$ cannot be a counterexample. Hence we can
find $\mu = \sum^{5}_{i=1}b_i\om_i$ with $\nu \succ \mu$ and 
$b_3 \geq 25$. But in this case we also have $\lam \succ \mu$, a contradiction.

Applying the aforementioned claim to $\lam_0$, we get 
$\mu = \sum^{5}_{i=1}b_i\om_i$ with $\lam_0 \succ \mu$ and 
$b_3 \geq 25$. Consider 
$$\beta := \al_2 +3\al_3 + \al_4 = 5\om_3 - \sum^{5}_{i=1}\om_i.$$
Then $\gam := \mu-5\beta = \sum^{5}_{i=1}c_i\om_i$ has all $c_i \geq 5$.
Now, for any $5$-tuple $(d_1,d_2,d_3,d_4,d_5) \in \ZZ^{5}$ with
$0 \leq d_i \leq 2$, notice that for  
$$\delta := \sum^{5}_{i=1}d_i\al_i = \sum^{5}_{i=1}e_i\om_i$$
we have $e_i = -d_{i-1}+2d_i-d_{i+1} \leq 4$ (where we set $d_0 = d_6 = 0$).
It follows that 
$\gam-\delta = \sum^{5}_{i=1}(c_i-e_i)\om_i$ is a good weight. Clearly,
$\lam_0 \succ \mu \succ \gam \succ \gam-\delta$. Thus for any $\delta$,
$\gam-\delta$ is a good weight which occurs in $L(\lam_0)$ by
Theorem \ref{premet}. As mentioned
before, the $W$-orbit of any good weight has length
$=|W| = 720$. Thus $n \geq \dim L(\lam_0) \geq 3^5 \cdot 720 = 174,960$, 
a contradiction.     
\end{proof}

\section{Characteristic $2$ case}

We will need the following well-known general fact:

\begin{lemma}\label{fixed}
{\sl Let $\GC$ be a simple algebraic group with a root system $\Phi$ with respect to 
a fixed maximal torus $\TC$. Let $\PC$ be a (proper) parabolic subgroup
of $\GC$, with unipotent radical $\QC$ and Levi subgroup $\LC$.
Assume that $g \in \GC$ normalizes $\LC$ but not $\PC$. Assume in addition
that either $\PC$ is a maximal parabolic subgroup, or 
$g \in N_{\GC}(\TC)$ induces a map sending the set of positive roots $\Phi^{+}$ to 
the set of negative roots $\Phi^{-}$. Then 

{\rm (i)} $\langle \QC,\QC^{g} \rangle = \GC$;

{\rm (ii)} If $V$ is any nontrivial, finite-dimensional, irreducible $\GC$-module,
then $\dim V^{\QC} \leq (\dim V)/2$.}
\end{lemma}

\begin{proof}
(i) Set $\RC := \langle \QC,\QC^g \rangle$. Then $\RC$ is normalized by 
$\LC = \LC^g$ and
by itself. It follows that $N_{\GC}(\RC)$ contains $\PC$ and $\PC^g$; in particular
it is a parabolic subgroup of $\GC$. Now if $\PC$ is a maximal parabolic subgroup,
we see that $N_{\GC}(\RC) = \GC$ as $\PC^g \neq \PC$. Assume $g \in N_{\GC}(\TC)$ 
sends $\Phi^{+}$ to $\Phi^{-}$. Then $g$ sends the Borel subgroup
$\langle \TC,U_\al \mid \al \in \Phi^{+}\rangle$ of $\PC$ to the opposite
Borel subgroup $\langle \TC,U_\al \mid \al \in \Phi^{-}\rangle$. Since 
$\GC = \langle \TC,U_\al \mid \al \in \Phi \rangle$, we also have
$N_{\GC}(\RC) = \GC$ in this case. Thus in either case $\RC \lhd \GC$, and $\RC$ has 
positive dimension. By simplicity of $\GC$, $\RC = \GC$.   

(ii) Clearly, every vector in $U := V^{\QC} \cap g(V^{\QC}) = V^{\QC} \cap V^{\QC^g}$ 
is fixed by both $\QC$ and $\QC^{g}$, and so it is fixed by 
$\langle \QC,\QC^g \rangle = \GC$. Since $V$ is irreducible and nontrivial, 
$U = 0$, and the claim follows. 
\end{proof}

Now for characteristic $2$ we can prove a much better bound for $R_n(\GC)$:

\begin{theor}\label{mainA2}
{\sl Let $\GC$ be a simple algebraic group over an algebraically closed field 
$\FF$ of characteristic $2$. Then $R_n(\GC) \leq n$.}
\end{theor}

\begin{proof}
1) First we consider the case $\GC$ is of type $A_r$.
Since $\Char(\FF) = 2$, there is a bijection $\lam \leftrightarrow J$ 
between the restricted dominant weights $\lam = \sum^{r}_{i=1}a_i\om_i$
and the sets $J = J(\lam) = \{ i \mid a_i = 0\}$. Now among all restricted 
$\lam$ with $\dim L(\lam) \leq n$, choose $\lam_0$ with 
$J_0 = J(\lam_0)$ of {\it smallest possible}
cardinality $m$; in particular, $0 \leq m \leq r$. Again by Lemma \ref{orb},
the stabilizer of $\lam_0$ in $W$ is the Young subgroup $\SSS_{J_0}$ 
generated by the reflections along the simple roots $\al_j$ with $j \in J_0$. 
It is easy to see that $(a+1)! \cdot (b+1)! \leq (a+b+1)!$ for non-negative 
integers $a,b$. In turn, this inequality implies that 
$|\SSS_{J_0}| \leq |\SSS_{m+1}| = (m+1)!$, whence
$$n \geq \dim L(\lam_0) \geq \frac{(r+1)!}{(m+1)!}.$$
By the choice of $\lam_0$, we see that $|J(\lam)| \geq m$ for all 
$\lam$ with $\dim L(\lam) \leq n$. It follows
that 
$$R_n \leq  \binom rr  + \binom r{r-1} + \cdots + \binom r{m+1} + \binom rm.$$
An easy induction on $r-m$ shows that the right-hand-side of the last 
inequality cannot exceed $\dfrac{(r+1)!}{(m+1)!}$.   

\medskip
2) In general, if $\GC$ is of bounded rank $r \leq 8$, then using 
\cite{Lu} for $n < 256$ and the obvious bound $R_n \leq 2^r$ for $n \geq 256$, 
we are done. 
So it remains to consider the case $\GC$ is of type $C_r$ or $D_r$ with 
$r \geq 9$, and with no loss we may assume $n > 2$. 
In either case, $\GC$ has a maximal parabolic subgroup with 
unipotent radical $\QC$ and Levi subgroup $\LC$ of 
type $A_{r-1}$, corresponding to, say, the simple root $\al_r$. Indeed, one can 
choose a hyperbolic basis $(e_1, \ldots ,e_r,f_1, \ldots ,f_r)$ of the natural module
$\FF^{2r}$ for $\GC$ and choose $\PC$ to be the stabilizer in $\GC$ of the maximal totally 
isotropic (singular in the case of type $D_r$) subspace 
$\langle e_1, \ldots ,e_r \rangle_{\FF}$. It is clear that there is an 
element $g \in \GC$ which sends each $e_i$ to
$f_i$ and vice versa. Thus $g$ normalizes $\LC$ but not $\PC$.     
Consider any restricted irreducible $\FF\GC$-module  
$V = L(\lam)$ with highest weight $\lam = \sum^r_{i=1}a_i\om_i$ and of dimension
$\leq n$. By Smith's Theorem
\cite{Sm}, the fixed-point subspace $V^{\QC}$ is an irreducible $\LC$-module 
$L'(\lam')$ with highest weight $\lam' = \sum^{r-1}_{i=1}a_i\om'_i$ 
(if we choose the fundamental 
weights $\om'_1, \ldots,\om'_{r-1}$ accordingly), and by Lemma \ref{fixed}, 
$\dim V^{\QC} \leq (\dim V)/2$ if $\lam \neq 0$. By 1), we have at most $n/2-1$ 
nonzero restricted dominant weights $\lam'$ such that $\dim L'(\lam') \leq n/2$. 
Since $\lam$ is completely determined by $\lam'$ and $a_r \in \{ 0,1\}$, we conclude that 
$$R_n \leq 2 (\frac{n}{2}-1) + 2 = n,$$
where the second summand, $2$, counts the weights $\lam$ with $\lam' = 0$.       
\end{proof}

\section{Other simple algebraic groups in odd characteristic}

Let $p > 2 $ be a prime and let $\GC$ be a simple (simply connected) algebraic group 
of rank $r$, not of type $A_r$. We fix a maximal torus $\TC$ in $\GC$,  
the set $\Delta:= \{ \al_1, \ldots, \al_r\}$ of simple roots, the corresponding set 
$\{ \om_1, \ldots ,\om_r\}$ of fundamental weights, and the set 
$$\LP = \left\{ \sum^{r}_{i=1}a_i \om_i \mid a_i \in \ZZ,a_i \geq 0
  \right\}$$
of dominant weights. Also let $\Gamma$ denote the Dynkin diagram for $\GC$, with simple
roots as vertices. Then there is a {\it positive} root $\al_0 = \sum^{r}_{i=1}n_i\al_i$ 
such that $\{-\al_0,\al_1, \ldots ,\al_r\}$ is the set of vertices for the 
extended Dynkin diagram $\Gamma^{(1)}$. Observe that, since $\GC$ is not
of type $A_r$, there is a simple root $\al_j$ such that $\al_0$ is connected only
to $\al_j$ in $\Gamma^{(1)}$. We will consider the following subsystem subgroup
$$\HC := \langle X_{\al_i},X_{-\al_i} \mid 0 \leq i \leq r,~i \neq j\rangle,$$
where as usual $X_{\beta}$ is the root subgroup corresponding to the root $\beta$.
Our observation implies that $\HC$ is the direct product $\HC_0 \times \HC_1$ of semisimple 
subgroups $\HC_0$ of type $A_1$ with simple root system $\{\al_0\}$, and 
$\HC_1$ with simple root system $\Delta_1 := \Delta \setminus \{\al_j\}$. 
We can decompose $\TC = \TC_0 \times \TC_1$, where $\TC_0$ is a maximal torus in $\HC_0$ and 
$\TC_1$ is a maximal torus in $\HC_1$. Then, without loss of generality,
we may identify the set of fundamental weights of $\HC_1$ with 
$\{\om_i \mid 1 \leq i \leq r,~i \neq j\}$. Let $\om_0$ denote the (unique) fundamental
weight for $\HC_0$.

\begin{lemma}\label{reduction}
{\sl Let $V = L(\om)$ be any restricted irreducible $\FF\GC$-module with highest weight
$\om = \sum^r_{i=1}a_i\om_i$. Then the following statements hold.

{\rm (i)} The $\HC$-module $V$ contains a simple subquotient 
isomorphic to $U_0 \otimes U_1$, where $U_0$ is the irreducible $\HC_0$-module with
highest weight $(\sum^r_{i=1}n_ia_i)\om_0$ and $U_1$ is the restricted 
irreducible $\HC_1$-module with highest weight $\sum_{1 \leq i \leq r,~i \neq j}a_i\om_i$.

{\rm (ii)} If in addition $n_j$ is coprime to $p$, then
$V$ is uniquely determined by the simple module $U_1$ and 
the remainder $m$ modulo $p$ of $\sum^r_{i=1}n_ia_i$. {\rm (In what follows
we will choose $m$ between $0$ and $p-1$.)}

{\rm (iii)} Assume $n_i \geq 1$ for all $i$ and $V$ is nontrivial. Then $\dim U_0 \geq m+1$
if $1 \leq m \leq p-1$ and $\dim U_0 \geq 2$ if $m = 0$.}
\end{lemma}

\begin{proof}
Let $v \in V$ be a highest weight vector. Then $v$ is fixed by any positive root
subgroup $X_{\beta}$, and $\TC$ acts on $v$ via the weight $\om$. In particular,
$v$ is fixed by any positive root subgroup of $\HC_0$, resp. of $\HC_1$, and
$\TC_1$ acts on $v$ via the weight $\om' := \sum_{1 \leq i \leq r,~i \neq j}a_i\om_i$. Furthermore,
since $\langle \om_i,\al_j \rangle = \delta_{ij}$ and $\langle \om_0,\al_0 \rangle = 1$,
we have 
$$\langle \om,\al_0 \rangle = \sum^r_{i=1}n_ia_i = 
  \langle (\sum^r_{i=1}n_ia_i)\om_0,\al_0 \rangle,$$
whence $\TC_0$ acts on $v$ via the weight $\om'' := (\sum^r_{i=1}n_ia_i)\om_0$.
Now let $W$ be the $\FF$-span of $\HC v$ and let $U$ be a maximal $\HC$-submodule 
of $W$. Then $W/U$ is an irreducible $\HC$-module, with high weight vector $v +U$. 
Now we can identify $W/U$ with $U_0 \otimes U_1$, 
where the irreducible $\HC_0$-module $U_0$ has
highest weight $\om''$ and the irreducible $\HC_1$-module $U_1$ 
has highest weight $\om'$.

Clearly, knowing the highest weight $\om'$ of $U_1$ allows us to recover all the 
coefficients $a_i$ of $\om$ with $i \neq j$. If in addition $n_j$ is coprime to $p$,
then knowing $m$ allows us to recover the remaining coefficient $a_j$ since
$0 \leq a_j \leq p-1$. 

For (iii), observe that $\dim U_0 \geq \dim L(m\om_0) = m+1$ by Steinberg's tensor product
theorem. Assume furthermore that $m = 0$. Then the nontriviality of $V$ implies that
$a_k \geq 1$ for some $k$. But $n_i \geq 1$ for all $i$, so $\sum^r_{i=1}n_ia_i > 0$. 
Hence $\dim U_0 \geq 2$ again by Steinberg's tensor product theorem. 
\end{proof}  

The rest of this section is devoted to proving the following theorem:

\begin{theor}\label{main2}
{\sl Let $\GC$ be a simple (simply connected) algebraic group in characteristic $p > 2$, 
not of type $A$. Then
$$R_n(\GC) \leq n^{2.5}.$$}
\end{theor}

\subsection{Type $C_r$}
We prove by induction on $r$ that $R_n(\GC) \leq n^2$, with the induction base $r=1$ 
being obvious as $R_n(\SL_{2}(\FF)) \leq n$. For the inductive step, label the simple roots
such that $\al_r$ is long and connected to $\al_{r-1}$ by a double edge in $\Gamma$. Notice
that $\al_0 = 2\sum^{r-1}_{i=1}\al_i+\al_r$, and the root $\al_j$ 
distinguished above is just $\al_1$, and so $n_j = 2$ and 
$n_i \geq 1$ for all $i$. Furthermore, $\HC_1$ is of type 
$C_{r-1}$. The statement is obvious when $n = 1$. 
Since the smallest dimension of nontrivial irreducible 
representations of $Sp_4(\FF)$ is $4$, cf. \cite{Lu}, we may assume that $n \geq 4$.

Consider any nontrivial restricted representation 
$V = L(\om)$ with $\dim V \leq n$. Since $n_j = 2$, 
we already know by Lemma \ref{reduction}
that $V$ is completely determined by $m$ and $U_1$, where $0 \leq m \leq p-1$ and clearly 
$$\dim U_1 \leq \frac{\dim V}{\dim U_0} \leq \frac{n}{\dim U_0}.$$
Applying induction hypothesis to $\HC_1$, we see that
the number of possibilities for $U_1$ is at most $(n/\dim U_0)^{2}$. It follows 
by Lemma \ref{reduction}(iii) that
$$R_n(\GC) \leq 1+\sum^{p-1}_{m=1}\left( \frac{n}{m+1} \right)^{2} + 
  \left( \frac{n}{2} \right)^{2} < 1 + n^{2} \cdot \left(\zeta(2) -1 + \frac{1}{4}\right) < n^2$$
since $n \geq 4$ and $\zeta(2) = \pi^2/6$, where 
$\zeta(s)=\sum^{\infty}_{k=1}k^{-s}$ is the Riemann zeta function. 
  
\subsection{Type $B_r$}
We prove by induction on $r \geq 2$ that $R_n(\GC) \leq n^{s}$ with $s = 9/4$.
The induction base $r = 1,2$ is clear since $B_1 \cong A_1$ and $B_2 \cong C_2$. 
For the inductive step when $r \geq 3$, label the simple roots such that $\al_r$ is short
and connected to $\al_{r-1}$ by a double edge in $\Gamma$. Notice
that $\al_0 = \al_1 +2\sum^{r}_{i=2}\al_i$, and the root $\al_j$ 
distinguished above is just $\al_2$; in particular $n_j = 2$ is coprime to $p$
and $n_i \geq 1$ for all $i$. Now
$\HC_1$ is of type $A_1 + B_{r-2}$. Since the 
smallest dimension of nontrivial irreducible 
representations of $\Spin_7(\FF)$ is $7$, cf. \cite{Lu}, we may assume that $n \geq 7$.

Consider any nontrivial restricted $V = L(\om)$ with $\dim V \leq n$. Since $n_j = 2$, 
by Lemma \ref{reduction} we have that $V$ is completely determined by $m$ and 
$U_1 = L(a_1\om_1) \otimes U_2$, where $0 \leq m \leq p-1$, $L(a_1\om_1)$ is the 
irreducible $\SL_2(\FF)$-representation with highest weight $a_1\om_1$, and $U_2$ is
the restricted irreducible $\Spin_{2r-3}(\FF)$-representation with highest weight 
$\sum^{r}_{i=3}a_i\om_i$. Clearly, 
$$\dim U_2 \leq \frac{\dim U_1}{a_1+1}\leq \frac{n}{(a_1+1)\dim U_0}.$$
By Lemma \ref{reduction}(iii), 
$\dim U_0 \geq m+1$ if $1 \leq m \leq p-1$, and $\dim U_0 \geq 2$ if $m = 0$.
Applying induction hypothesis to the factor $B_{r-2}$ of $\HC_1$, we see that
the number of possibilities for $U_2$ is at most $(n/(a_1+1)\dim U_0)^{s}$. 
It follows that
$$\begin{array}{ll}R_n(\GC) & \leq 1+\sum^{p-1}_{a_1 =0}\sum^{p-1}_{m=1}
  \left( \dfrac{n}{(a_1+1)(m+1)} \right)^{s} + 
  \sum^{p-1}_{a_1 =0}\left( \dfrac{n}{2(a_1+1)} \right)^{s} \\ \\
  & < 1 + n^{s} \cdot \left(\zeta(s)(\zeta(s) -1) + 
  \dfrac{\zeta(s)}{2^{s}}\right) < n^s \end{array}$$
since $n \geq 7$ and $s=9/4$. 

\subsection{Type $D_r$}
We prove by induction on $r \geq 2$ that $R_n(\GC) \leq n^{s}$ with $s = 9/4$.
In the case $r = 2$, $\GC \cong \SL_2(\FF) \times \SL_2(\FF)$ and so we have the 
obvious bound $R_n \leq n^2$. We show that $R_n \leq n^2$ also holds in the case $r=3$, 
where $\GC \cong \SL_4(\FF)$. If $n \leq 23$ then the bound can be verified by 
inspecting \cite{Lu}. If $n \geq 24$, then (\ref{bound2}) implies
$$R_n \leq 2(n+3)(1+\log((n+3)/4))^2 < n^2.$$
For the inductive step when $r \geq 4$, label the simple roots such that $\al_{r-2}$ 
is connected to $\al_{r-3}$ and the two end nodes $\al_{r-1}$ and $\al_r$ in $\Gamma$. 
Notice that 
$$\al_0 = \al_1+\al_{r-1}+\al_r + 2\sum^{r-2}_{i=2}\al_i,$$
in particular, $n_i \geq 1$ for all $i$.  
The root $\al_j$ distinguished above is $\al_2$ and so $n_j = 2$, and 
$\HC_1$ is of type $A_1 + D_{r-2}$. Since the 
smallest dimension of nontrivial irreducible 
representations of $\Spin_8(\FF)$ is $8$, cf. \cite{Lu}, we may assume that $n \geq 8$.

Consider any nontrivial restricted $V = L(\om)$ with $\dim V \leq n$. Since $n_j = 2$, 
by Lemma \ref{reduction} we have that $V$ is completely determined by $m$ and 
$U_1 = L(a_1\om_1) \otimes U_2$, where $0 \leq m \leq p-1$, $L(a_1\om_1)$ is the 
irreducible $\SL_2(\FF)$-representation with highest weight $a_1\om_1$, and $U_2$ is
the restricted irreducible $\Spin_{2r-4}(\FF)$-representation with highest weight 
$\sum^{r}_{i=3}a_i\om_i$. Again we have
$$\dim U_2 \leq \frac{\dim U_1}{a_1+1}\leq \frac{n}{(a_1+1)\dim U_0}.$$
Now we can apply the induction hypothesis to the factor $D_{r-2}$ of $\HC_1$ and 
proceed as in the case of type $B_r$. 

\subsection{Exceptional groups}

\subsubsection{Type $E_6$}
We prove that $R_n(\GC) \leq n^{2.5}$. Label the simple roots
such that $\al_4$ is connected to $\al_3$, $\al_5$, and the end node $\al_{2}$ in 
$\Gamma$. Then the root $\al_j$ distinguished above is $\al_2$, and $\HC_1$ is of type 
$A_5$. Notice that 
$$\al_0 = \al_1+2\al_2+2\al_3+3\al_4+2\al_5+\al_6$$ 
so that $n_i\geq 1$ for all $i$ and $n_j =2$.  
Since the smallest dimension of nontrivial irreducible 
representations of $\GC$ is $27$, cf. \cite{Lu}, we may assume that $n \geq 27$.

Consider any nontrivial restricted $V = L(\om)$ with $\dim V \leq n$. Since $n_j = 2$, 
by Lemma \ref{reduction} we have that $V$ is uniquely determined by $m$ and $U_1$, 
where $0 \leq m \leq p-1$ and $\dim U_1 \leq n/(\dim U_0)$. Also,
$\dim U_0 \geq m+1$ if $1 \leq m \leq p-1$ and $\dim U_0 \geq 2$ if $m = 0$
by Lemma \ref{reduction}(iii). 
Applying Proposition \ref{A5} to $\HC_1 \cong \SL_6(\FF)$, we see that
the number of possibilities for $U_1$ is at most $(n/\dim U_0)^{2.5}$. It follows that
$$R_n(\GC) \leq 1+\sum^{p-1}_{m=1}\left( \frac{n}{m+1} \right)^{2.5} + 
  \left( \frac{n}{2} \right)^{2.5} < 
  1 + n^{2.5} \cdot \left(\zeta(2.5) -1 + \frac{1}{2^{2.5}}\right) < n^{2.5}$$
since $n \geq 27$.

\subsubsection{Type $E_7$}
We prove that $R_n(\GC) \leq n^{9/4}$. Label the simple roots
such that $\Delta \setminus \{\al_7\}$ is of type $E_6$ and $\al_1, \ldots ,\al_6$ 
are labeled as in the case $E_6$. Then the root $\al_j$ distinguished above is 
$\al_1$, and $\HC_1$ is of type $D_6$. Notice that 
$$\al_0 = 2\al_1+2\al_2+3\al_3+4\al_4+3\al_5+2\al_6+\al_7$$ 
so that $n_i\geq 1$ for all $i$ and $n_j =2$.  
Since the smallest dimension of nontrivial irreducible 
representations of $\GC$ is $56$, cf. \cite{Lu}, we may assume that $n \geq 56$.

Consider any nontrivial restricted $V = L(\om)$ with $\dim V \leq n$. Since $n_j = 2$, 
by Lemma \ref{reduction} we have that $V$ is uniquely determined by $m$ and $U_1$, 
where $0 \leq m \leq p-1$ and $\dim U_1 \leq n/(\dim U_0)$. 
Applying the proved bound for $\HC_1$ of type $D_6$, we see that
the number of possibilities for $U_1$ is at most $(n/\dim U_0)^{9/4}$. It follows that 
$$R_n(\GC) \leq 1+\sum^{p-1}_{m=1}\left( \frac{n}{m+1} \right)^{9/4} + 
  \left( \frac{n}{2} \right)^{9/4} < 
  1 + n^{9/4} \cdot \left(\zeta(9/4) -1 + \frac{1}{2^{9/4}}\right) < n^{9/4}$$
since $n \geq 56$.

\subsubsection{Type $E_8$}
We prove that $R_n(\GC) \leq n^{9/4}$. Label the simple roots
such that $\Delta \setminus \{\al_8\}$ is of type $E_7$ and $\al_1, \ldots ,\al_7$ 
are labeled as in the case $E_7$. Then the root $\al_j$ distinguished above is 
$\al_8$, and $\HC_1$ is of type $E_7$. Notice that 
$$\al_0 = 2\al_1+3\al_2+4\al_3+6\al_4+5\al_5+4\al_6+3\al_7+2\al_8$$ 
so that $n_i\geq 1$ for all $i$ and $n_j =2$.  
Since the smallest dimension of nontrivial irreducible 
representations of $\GC$ is $248$, cf. \cite{Lu}, we may assume that $n \geq 248$.
Arguing as above and applying the proved bound for $\HC_1$ of type $E_7$, we see that
$$R_n(\GC) \leq 1+\sum^{p-1}_{m=1}\left( \frac{n}{m+1} \right)^{9/4} + 
  \left( \frac{n}{2} \right)^{9/4} < 
  1 + n^{9/4} \cdot \left(\zeta(9/4) -1 + \frac{1}{2^{9/4}}\right) < n^{9/4}$$
since $n \geq 248$.

\subsubsection{Type $F_4$}
We prove that $R_n(\GC) \leq n^{2}$. Label the simple roots
such that $\al_2$ is long and connected to the end node $\al_1$ and also to
the short root $\al_3$ by a double edge in $\Gamma$. Then the root $\al_j$ distinguished 
above is $\al_1$, and $\HC_1$ is of type $C_3$. Notice that 
$$\al_0 = 2\al_1+3\al_2+4\al_3+2\al_4$$ 
so that $n_i\geq 1$ for all $i$ and $n_j =2$.  
Since the smallest dimension of nontrivial irreducible 
representations of $\GC$ is $25$, cf. \cite{Lu}, we may assume that $n \geq 25$.
Arguing as above and applying the proved bound for $\HC_1$ of type $C_3$, we see that
$$R_n(\GC) \leq 1+\sum^{p-1}_{m=1}\left( \frac{n}{m+1} \right)^{2} + 
  \left( \frac{n}{2} \right)^{2} < 
  1 + n^{2} \cdot \left(\zeta(2) -1 + \frac{1}{4}\right) < n^{2}$$
since $n \geq 25$.

Since $R_n(\GC) \leq n^2$ in the case $\GC$ is of type $G_2$, we have completed the proof
of Theorem \ref{main2}.
\hfill $\Box$

\section{Symmetric and alternating groups}

Recall that the $p$-modular irreducible representations of the symmetric group
$\SSS_{r}$ are parametrized by the $p$-regular partitions of $r$ (i.e. the partitions
$\lam = (\lam_1, \ldots ,\lam_s) \vdash r$ with $\lam_1 \geq \ldots \geq \lam_s \geq 1$,
$\sum^{s}_{i=1}\lam_i = r$, and for each $j$ there are at most $p-1$ indices $i$ 
such that $\lam_i = j$). For such a $\lam$ denote the corresponding 
irreducible $\FF\SSS_r$-module 
by $D^{\lam}$. It was shown by James in \cite{J} that  
$\dim D^{\lam} \approx (r^m/m!)\dim D^{\bar{\lam}}$ when $r \to \infty$, where 
$\bar{\lam} := (\lam_2, \ldots ,\lam_s)$; however the resulting lower bound for 
$\dim D^{\lam}$ is ineffective. 

To estimate the (modular) representation growth for 
$\SSS_{r}$ we need the following effective bound, which is also of independent interest.
For $\lam$ of indicated shape, define $m_2(\lam) := \lam_1$. For $p \neq 2$,
consider the sign representation ${\bf {sgn}}$ of $\SSS_r$ over $\FF$. 
Mullineux defined a bijection
$\MB$ on the set of $p$-regular partitions of $r$ and conjectured in \cite{M} that 
$D^{\lam} \otimes {\bf {sgn}} \cong D^{\lam^{\MB}}$. Mullineux's conjecture was  
proved by Ford and Kleshchev in \cite{FK}, and independently by Bessenrodt and Olsson
in \cite{BO}. Denoting the first (largest) part of $\lam^{\MB}$ by 
$(\lam^{\MB})_1$, we define     
$$m_p(\lam):= \max(\lam_1,(\lam^{\MB})_1),$$ 
(that is, the longest ``row'' of $\lam$ and $\lam^{\MB}$).   

\begin{theor}\label{bound3}
{\sl For any $n \geq 5$, any $p \geq 0$, and any $p$-regular partition $\lam$,
$$\dim D^{\lam} \geq 2^{\frac{r-m_p(\lam)}{2}}.$$}
\end{theor}

\begin{proof}
1) We proceed by induction on $r$, with the induction base $r = 5,6$ easily checked using 
\cite{JLPW}. For the induction step, assume $r \geq 7$ and consider any 
$p$-regular partition $\lam \vdash r$. The statement is obvious if 
$m_p(\lam) = r$. So we will assume $m_p(\lam) \leq r-1$, which implies 
that $\dim D^{\lam} \geq 3$. (For, suppose that $\dim D^{\lam} \leq 2$. Then by
the results of \cite{J}
we see that $\dim D^{\lam} = 1$ and, moreover,
either $\lam = (r)$ and so $m_p(\lam) = r$, or 
$p \neq 2$ and $D^{\lam} \cong {\bf {sgn}}$. In the latter case,
$D^{\lam^{\MB}}$ is the trivial representation, whence $\lam^{\MB} = (r)$ and so
$m_p(\lam) = r$ as well.)
We will also consider a chain of 
natural subgroups $K < H < G = \SSS_r$, with $H = \SSS_{r-1}$ and $K = \SSS_{r-2}$.
Our proof will rely on the beautiful result of
Kleshchev describing branching rules for modular representations of $\SSS_r$, cf. 
\cite{K}. In particular, this result says that the Young diagram $Y(\lam)$ of
$\lam$ has (at most $p-1$) {\it good} nodes $A_{1}, \ldots ,A_a$ such that 
$$\soc\left(D^{\lam} \downarrow_{H}\right) = \oplus^{a}_{i=1}D^{\lam^i},$$
where the Young diagram of the ($p$-regular) partition $\lam^i \vdash (n-1)$ 
is obtained from $Y(\lam)$ by removing the node $A_i$ for $1 \leq i \leq a$, and similarly 
for $\soc\left(D^{\lam^i} \downarrow_{K}\right)$. 

2) We claim that $m_p(\lam^i) \leq m_p(\lam)$. This is obvious for $p = 2$, since
node removal does not increase the largest length of the rows of the partition. 
Consider the case $p \neq 2$. Then again we have that the longest row
$(\lam^i)_1$ of $\lam^i$ is at most $\lam_1 \leq m_p(\lam)$. Next, if $\Phi$ denotes
a particular matrix representation of $\SSS_r$ afforded by $D^{\lam}$, then for any 
$g \in \SSS_r$, $g$ acts on $D^{\lam^{\MB}}$ via the matrix ${\bf {sgn}}(g)\Phi(g)$. Since 
$D^{(\lam^i)^{\MB}} \cong D^{\lam^i} \otimes {\bf {sgn}}$, it follows that
$$\soc\left(D^{\lam^{\MB}} \downarrow_{H}\right) = \oplus^{a}_{i=1}D^{(\lam^i)^{\MB}}.$$
By the aforementioned result of Kleshchev applied to $D^{\lam^{\MB}}$, 
$Y((\lam^i)^{\MB})$ is obtained by removing 
a good node $B_i$ from $Y(\lam^{\MB})$. Therefore the longest row
$((\lam^i)^{\MB})_1$ of $(\lam^i)^{\MB}$ is at most $(\lam^{\MB})_1 \leq m_p(\lam)$. Thus
$m_p(\lam^i) \leq m_p(\lam)$ as stated.
 
3) First we assume that the $K$-module $D^{\lam} \downarrow_{K}$ has
at least two composition factors $D^{\mu^1}$ and $D^{\mu^2}$, 
where each $Y(\mu^j)$ is obtained from 
$Y(\lam)$ by removing two good nodes subsequently. By the result of 2) we have 
$m_p(\mu^j) \leq m_p(\lam)$. Applying the 
induction hypothesis to $\mu^j$ we obtain
$$\dim D^{\lam} \geq \dim D^{\mu^1} + \dim D^{\mu^2} \geq 
  2^{\frac{(r-2)-m_p(\mu^1)}{2}} + 2^{\frac{(r-2)-m_p(\mu^2)}{2}} \geq 
  2 \cdot 2^{\frac{(r-2)-m_p(\lam)}{2}} = 2^{\frac{r-m_p(\lam)}{2}},$$
as desired.

Next suppose that either $\lambda$ has at least two good nodes
$A_1$, $A_2$, or $\lambda$ has only one good node $A_1$ but 
$\lam^1$ has at least two good nodes. Then Kleshchev's branching rule tells us 
that $\soc\left(D^{\lam} \downarrow_{K}\right)$ contains a submodule
$D^{\mu^1} \oplus D^{\mu^2}$, where each $Y(\mu^j)$ is obtained from $Y(\lam)$ by 
removing two good nodes, and so we are done.

Now we may assume that $\lam$ has only one good node $A_1$ and $\lam^1$ has only one good 
node $B_1$. In particular, $D^{\lam} \downarrow_{H}$ is an indecomposable module with 
simple socle $D^{\lam^1}$. But this module is self-dual, so its head is also simple and 
isomorphic to $D^{\lam^1}$. Assuming $D^{\lam} \downarrow _{H}$ is not simple, we see
that $D^{\lam} \downarrow _{H}$ has composition factor $D^{\lam^1}$ with multiplicity
$\geq 2$. Hence, $D^{\lam} \downarrow _{K}$ has a composition factor $D^{\mu^1}$ 
with multiplicity $\geq 2$, where $Y(\mu^1)$ is obtained from $Y(\lam)$ by removing 
two good nodes, and so we are again done. In the remaining case, we have that 
$D^{\lam}$ is irreducible over $H$. Applying the same argument to 
$D^{\lam} \downarrow_K$, we are done if $D^{\lam} \downarrow_K$ is not simple. 
We can now complete the induction step by observing that, since 
$G > K \times C$ with $C = \SSS_2 \not\leq Z(G)$ and $\dim D^{\lam} > 1$, 
$D^{\lam} \downarrow_K$ can never be simple.    
\end{proof} 

Theorem \ref{bound3} yields the following sublinear bound for $R_n(\SSS_r)$:

\begin{corol}\label{sym1}
{\sl If $r \geq 5$, $n \geq 2$, and $\Char(\FF) = p \geq 0$, then 
$$R_{n}(\SSS_r) \leq f_{5}(n) := 4(\log_2 n) 
  e^{2\pi\sqrt{(\log_2 n)/3}}.$$}
\end{corol}

\begin{proof}
Consider any $p$-regular $\lam \vdash r$ with $\dim D^{\lam} \leq n$. Then, by 
Theorem \ref{bound3}, $n \geq 2^{(r-m_p(\lam))/2}$ and so $m_p(\lam) \geq r-2\log_2 n$. 
Replacing $\lam = (\lam_1, \ldots ,\lam_s)$ by $\lam^{\MB}$ if 
necessary, we may assume that 
$\lam_1 = m_p(\lam)$. Thus $\bar\lam := (\lam_2, \ldots, \lam_s)$ is a partition 
of $r - \lam_1 = r -m_p(\lam) \leq n_0 := \lfloor 2\log_2 n \rfloor$. The number of 
possible $\bar\lam$ is at most $p(r-\lam_1)$. Hence
$$R_n \leq 2(p(0) + p(1) + \cdots + p(n_0)).$$ 
Observe that $p(0) + p(1) = p(2) \leq p(n_0)$ since $n_0 \geq 2$. It follows that 
$R_n \leq 2n_0 \cdot p(n_0)$, and we are done by using the upper bound (\ref{pi})
for $p(n_0)$.    
\end{proof}
 
The function $f_5(n)$, even though sublinear in $n$, still involves large constants
and becomes small only when $n$ is very large. For instance, one can check that
$f_5(n) \leq n$ when $n \geq 10^{13}$ and
$f_5(n) < \sqrt{n}$ when $n \geq 10^{44}$. 
Nevertheless, using Corollary \ref{sym1} we can
prove the following practical bound:

\begin{theor}\label{sym-alt}
{\sl Let $G$ be a universal covering group of $\SSS_r$ or $\AAA_r$ with $r \geq 5$ and let 
$\Char(\FF) = p \geq 0$ as above. Then $$R_n(G) \leq n^{2.5}.$$}
\end{theor} 

\begin{proof}
1) Without loss we may assume that $n \geq 2$. The cases $5 \leq r \leq 12$ can be 
checked directly using \cite{JLPW}, so we will assume that $r \geq 13$. 
Let $d$ be the smallest dimension of 
nontrivial irreducible projective $p$-modular representations of $\AAA_r$. Clearly, 
$R_n(G) \leq 2 \leq n$ if $2 \leq n < d$, so we may assume that $n \geq d$. 
By \cite{J} and \cite{KT}, $d \geq r-2 \geq 11$; in particular, $n \geq 11$.

Consider the case $n \geq 2^{\lfloor (r-3)/2 \rfloor}$. Note that 
$R_n(G) \leq |\Irr(G)| \leq 4p(r)$. Since $r \geq 13$, one can check 
that $4p(r) < n^{2.5}$ (using (\ref{pi}) for $r \geq 19$ and directly for 
$13 \leq r \leq 18$). Thus we may assume that $n <  2^{\lfloor (r-3)/2 \rfloor}$.
By \cite[Theorem A]{KT}, this condition implies that any irreducible representation
of $G$ of dimension at most $n$ is in fact a representation of $G/Z(G)$. So in what
follows we may assume that $G \in \{\SSS_r,\AAA_r\}$.  

\smallskip
2) Now consider the case $G = \SSS_r$ with $r \geq 13$, $n \geq 11$. Claim that 
$$R_n(G) \leq b(n) := \frac{n^{2.5}}{12.32}.$$ 
By Corollary \ref{sym1}, $R_n(G) \leq f_5(n)$. If $n \geq 1503$, then one  
can check that $f_5(n) < b(n)$, and so we are done. On the other hand,
if $n < (r^2-5r+2)/2$, then $R_n(G) \leq 4$ by \cite{J}, and we are done again.
So we will assume that $n \geq (r^2-5r+2)/2$; in particular, $n \geq 53$. 
Now if $1503 > n \geq 677$, then
$r \leq 60$ and so
$$R_n(G) \leq p(60) = 966,467 < b(677) \leq b(n).$$
If $677 > n \geq 172$, then $r \leq 39$ and so
$$R_n(G) \leq p(39) = 31,185 < b(172) \leq b(n).$$
If $172 > n \geq 53$, then $r \leq 21$ and so
$$R_n(G) \leq p(21) = 792 < b(53) \leq b(n).$$

\smallskip
3) Finally, we consider the case $S = \AAA_r$ with $r \geq 13$. Consider any
$V \in \IBR_p(S)$ with $\dim V \leq n$. If $V$ extends to $G = \SSS_r$, then $V$ can
be obtained by restricting an irreducible $\FF G$-representation of degree $\leq n$ 
to $S$. On the other hand, if $V$ does not extend to $G$, then $V$ is one of the two
irreducible constituents of the restriction to $S$ of an irreducible 
$\FF G$-representation of degree $\leq 2n$. It follows by the main result of 2) that 
$$R_n(S) \leq R_n(G) + 2R_{2n}(G) \leq b(n) + 2b(2n) 
  = \frac{n^{2.5}}{12.32} + \frac{2(2n)^{2.5}}{12.32} < n^{2.5}.$$     
\end{proof}

Now Theorem \ref{main1} immediately follows from Theorems \ref{mainA}, \ref{mainA2}, 
\ref{main2}, and \ref{sym-alt}.

\section{Maximal Subgroups} \label{maximal}

In this section we will get upper bounds on the number of conjugacy classes of maximal
subgroups of finite almost simple groups.

We first recall two  results of Liebeck, Martin and Shalev \cite{LMS}.  

\begin{lemma} \label{altmax} 
{\rm \cite[Cor. 5.3]{LMS}}  
{\sl Let $G=\SSS_n$ or $\AAA_n$.   The number of conjugacy classes
of maximal subgroups of $G$ is $n^{1 + o(1)}$.}
\end{lemma}

\begin{lemma} \label{chevmax1}  
{\rm \cite[Theorem 1.3]{LMS}}   
{\sl There exists a function $c(r)$ and an absolute
constant $d$ such that   if $G$ is an almost simple group with
socle of a finite simple group of Lie type of rank $r$, then the number of conjugacy classes
of maximal subgroups is at most $c(r) + dr \log\log q$.}  
\end{lemma}

The major contribution of the $dr \log\log q$ factors comes from subfield groups.  

We use our results on the number of restricted representations of dimension at most $n$ together
with the results of Fulman and Guralnick \cite[\S7]{FG} to show that we may take $c(r)$ to
be a polynomial in $r$.   We may assume that $r$ is sufficiently large and so it suffices to consider
the classical Chevalley groups.     So we prove the following result which then implies Theorem \ref{main3}. 
Theorem \ref{main4} follows from Theorem \ref{main3} and Lemma \ref{altmax}.  

\begin{theorem}  \label{classicalmax}  
{\sl There exist absolute 
constants $a$ and $d$ such that if  $G$ be an almost simple   group with socle a classical
group of rank $r$ over a field $\FF_q$,  then the number of conjugacy classes of maximal subgroups is at most 
$ar^{6} + dr \log\log q$.} 
\end{theorem}

\begin{proof}  
We follow the proof in \cite{FG} except that in one case we use the results of this paper.
By Aschbacher's Theorem on maximal subgroups of classical groups \cite{As}, there
are $8$ ``natural''  families of  maximal subgroups and one further one consisting of maximal embeddings
of almost simple groups  into $G$.

It is well known that the number of conjugacy classes of maximal subgroups 
in the natural families is at most  $8r \log r + r \log\log q$
(see \cite{GKS, LPS}).   Thus, we need only consider maximal embeddings of almost simple groups.  In particular, 
these will correspond to absolutely irreducible representations of these groups.    We need not worry
about sporadic groups since for $r$ sufficiently large there will be no such representations. 

The natural module for $G$ is of dimension $r+1 ,2r$ or $2r+1$.   
Let $\FF_q$ be the field of definition
of $G$ (which is also the field of definition for the natural module aside from the case
$G={\mathrm {SU}}_{r+1}(q)$ where the field of definition for the module is $\FF_{q^2}$). 

As in \cite{FG}, we subdivide these maximal subgroups $M$ into four cases:
\begin{enumerate}
\item[$\SC1$]  the socle of $M$ is an alternating group;
\item[$\SC2$]  the socle of $M$ is a group of Lie type in characteristic different from $G$;
\item[$\SC3$]  the socle of $M$ is a group of Lie type in the same characteristic as $G$
and the representation is not restricted; or
\item[$\SC4$]  the socle of $M$ is a group of Lie type in the same characteristic as $G$
and the representation is  restricted.
\end{enumerate}

The cases $\SC2$ and $\SC3$ were already finished in \cite{FG} where upper bounds of $O(r^3)$
and  $O(r \log r)$ were established.   

First consider $\SC1$.   By Theorem \ref{main1}, the number of 
irreducible projective representations of 
dimension at most $2r+1$ for $\AAA_d$ is at most $(2r+1)^{2.5}$.  
The number of non-isomorphic alternating
groups $\AAA_d$ that have an irreducible projective representation of dimension $n$ is 
at most $n$ (since $5 \leq d \leq n+4$).   Thus, the total number of
projective irreducible representations of alternating groups is at most $ar^{3.5}$ for some constant $a$.
In the full group of isometries of the given classical groups, the (quasi)-equivalence class determines
a unique conjugacy class.  In order to account for this class splitting further, we need to multiply
this by $r$ (only for groups of type $A$ -- for the other groups, there is a small absolute constant).  Thus,
the total number of conjugacy classes of maximal subgroups in this case is at most $ar^{4.5}$.

Finally, consider $\SC4$.   So the socle of $M$ is a Chevalley group $H_s(q')$,
where $s$ is the rank and
$\FF_{q'}$ is the field of definition.   The field of definition for the representation must be the same as for
the field of definition of the natural module for $G$ (or $M$ would be contained in a subfield group).
Thus, there are at most two choices for $q'$.   Clearly, $s \le r$.   Thus, there are most $9r$ possible
choices for the socle of $M$. 
(One can improve this considerably; if $s$ is large (say $s > r^{1/2}$),
then  one can enumerate all the representations \cite{Lu}).   Thus, combining this with Theorem \ref{main1}
and arguing as in the previous paragraph, we see that the total number of conjugacy classes of maximal subgroups 
in this family is at most $ar^6$.
\end{proof}

\end{document}